\documentclass{amsart}

\usepackage{amsmath, amsthm, amscd, amsfonts, amssymb, graphicx, color}
\usepackage[bookmarksnumbered, colorlinks, plainpages]{hyperref}

\newcommand{\del}{\mbox{del}\,}

\newcommand{\lk}{\mbox{lk}\,}

\renewcommand{\dim}{\mbox{dim}\,}
\renewcommand{\dim}{\mbox{dim}\,}

\newcommand{\reg}{\mbox{reg}\,}

\newcommand{\T}{\mathrm}

\newtheorem{theorem}{Theorem}[section]
\newtheorem{corollary}[theorem]{Corollary}
\newtheorem{lemma}[theorem]{Lemma}
\newtheorem{proposition}[theorem]{Proposition}
\newtheorem{definition}[theorem]{Definition}

\newtheorem{example}[theorem]{Example}
\newtheorem{remark}[theorem]{Remark}

\numberwithin{equation}{section}

\begin{document}
\bibliographystyle{amsplain}

\title[Matchings in hypergraphs and Castelnuovo-Mumford regularity]{Matchings in hypergraphs and Castelnuovo-Mumford regularity}
\author[F. Khosh-Ahang and S. Moradi]{Fahimeh Khosh-Ahang$^*$ and Somayeh Moradi}
\address{Fahimeh Khosh-Ahang, Department of Mathematics,
 Ilam University, P.O.Box 69315-516, Ilam, Iran.}
\email{fahime$_{-}$khosh@yahoo.com}
\address{Somayeh Moradi, Department of Mathematics,
 Ilam University, P.O.Box 69315-516, Ilam, Iran and School of Mathematics, Institute
 for Research in Fundamental Sciences (IPM), P.O.Box: 19395-5746, Tehran, Iran.} \email{somayeh.moradi1@gmail.com}

\keywords{edge ideal, hypergraph, matching number, regularity, vertex
decomposable.\\
$*$Corresponding author}
\begin{abstract}

\noindent  In this paper, we introduce and generalize some combinatorial invariants of graphs such as matching number and induced matching number to hypergraphs. Then we compare them together and present some upper bounds for the regularity of Stanley-Reisner ring of $\Delta_{\mathcal{H}}$ for certain hypergraphs $\mathcal{H}$ in terms of the introduced matching numbers.
\end{abstract}

\maketitle
\section*{Introduction}

There is a natural correspondence between simplicial complexes and hypergraphs in the way that for a hypergraph $\mathcal{H}$, the faces
of the simplicial complex associated to it are the independent sets of vertices of $\mathcal{H}$, i.e. the sets which do not contain any edge of $\mathcal{H}$. This simplicial complex is called the \textbf{independence complex} of $\mathcal{H}$ and is denoted by $\Delta_{\mathcal{H}}$. Squarefree monomial ideals can be studied using these combinatorial ideas. Recently, edge ideals of graphs, as the easiest class of squarefree monomial ideals, has been studied by many researchers and some nice characterizations of the algebraic invariants, in terms of data from graphs, have been proved (cf. \cite{HD}, \cite{Kimura}, \cite{KM}, \cite{Mor}, \cite{VT}
and \cite{Zheng}). Extending the concepts in graphs to hypergraphs and finding more general results in hypergraphs, which will cover all
 squarefree monomial ideals, are of great interest and in some senses there are generalizations, see for example \cite{Emt}, \cite{HT1}, \cite{HW},  \cite{MVi} and  \cite{Wood}. The matchings are some graph invariants which are studied extensively (cf. \cite{LP}). In this paper we are going to extend some of them to hypergraphs.

The \textbf{Castelnuovo-Mumford regularity} (or simply regularity) of an $R$-module $M$ is defined as
$$\reg(M) := \max\{j-i |\
\beta_{i,j}(M)\neq 0\},$$
where $\beta_{i,j}(M)$ is the $(i,j)$th Betti number of $M$. Explaining the Castelnuovo-Mumford regularity of $R/I_{\Delta_{\mathcal{H}}}$ in terms of invariants of $\mathcal{H}$ has been studied extensively by many authors, where $I_{\Delta_\mathcal{H}}$ is the Stanley-Reisner ideal of the independence complex of the hypergraph $\mathcal{H}$. In the case that $\mathcal{H}$ is a graph, in certain circumstances, $\T{reg}(R/I_{\Delta_{\mathcal{H}}})$ is characterized precisely. For instance, in \cite{HT1}, \cite{KM} and \cite{VT}, respectively for chordal graph, $C_5$-free vertex decomposable graph and sequentially Cohen-Macaulay bipartite graph $G$, it was shown that $\reg(R/I(G))=c_G$, where $I(G)$ is the edge ideal of $G$ and $c_G$ is the induced matching number of $G$. Furthermore, combinatorial characterizations of the Castelnuovo-Mumford regularity of the edge ideal of hypergraphs has been subject of many works. Indeed, in \cite{HW}, the authors introduced the concept of
 2-collage in a simple hypergraph as a generalization of the matching number in graph and proved that
the Castelnuovo-Mumford regularity of the edge ideal of a simple hypergraph is bounded above
in terms of 2-collages. Also, Morey and Villarreal, in \cite{MVi}, gave a lower bound for the regularity of
the edge ideal of any simple hypergraph in terms of an induced matching of the hypergraph. Moreover, in \cite{HT1}, for $d$-uniform properly-connected hypergraphs a lower bound for the regularity is given. For more results see \cite{C,DHS,FV,N,WW}.

In this paper, we also study the regularity of the Stanley-Reisner ring of $\Delta_{\mathcal{H}}$ for some families of hypergraphs
 and relate it to some combinatorial concepts and generalize or improve some results, which had been gained for graphs, such as \cite[Theorem 6.7]{HT1} and \cite[Theorem 2.4]{KM}.

The paper proceeds as follows. After reviewing some hypergraph terminologies in the first section, in Section 2, we define an induced matching, a semi induced matching and matching number for a hypergraph $\mathcal{H}$, which we denote by $c_{\mathcal{H}}$, $c'_{\mathcal{H}}$ and $m_{\mathcal{H}}$, respectively and compare them together under different conditions. Also, we present a class of hypergraphs $\mathcal{H}$, consisting simple graphs, so that $c_{\mathcal{H}}=c'_{\mathcal{H}}$.

In the light of \cite[Corollary 3.9(a)]{MVi}, $c_\mathcal{H}$ is a lower bound for $\reg(R/I_{\Delta_{\mathcal{H}}})$, when $\mathcal{H}$ is a hypergraph. In Section 3, we are going to obtain some upper bounds for $\reg(R/I_{\Delta_{\mathcal{H}}})$ for a hypergraph $\mathcal{H}$.
As another class of hypergraphs, vertex decomposable hypergraphs has been studied and in Theorem \ref{reg}, it is proved that if a vertex
decomposable hypergraph $\mathcal{H}$ is $(C_2,C_5)$-free, then $\reg(R/I_{\Delta_{\mathcal{H}}})\leq c'_{\mathcal{H}}\leq \dim \Delta_{\mathcal{H}}+1$. This improves a result
on graphs proved in \cite{KM}, which states that for a $C_5$-free vertex decomposable graph $G$, $\reg(R/I(G))=c_G$.

\section{Review of hypergraph terminology}
In this section, we present some preliminaries in the context of hypergraphs from \cite{Berge} and \cite{Berge2}.

\begin{definition}
{\rm A \textbf{hypergraph} is a pair $(V,\mathcal{E}, I)$, where $V$ is a finite set of vertices, and $\mathcal{E} = \{E_i : i \in
I,  \emptyset\neq E_i \subseteq V \}$ is a collection of edges (or hyperedges). We will often abuse notation and refer
to $(V,\mathcal{E})$ as a hypergraph, with an understanding that the edges are indexed by
some set $I$. A hypergraph is called \textbf{$d$-uniform} if all of its edges have the same cardinality $d$. So, every simple graph is a $2$-uniform hypergraph.

Throughout this paper, we assume that $\mathcal{H}=(V(\mathcal{H}), \mathcal{E}(\mathcal{H}))$ is a \textbf{simple hypergraph}. That means that no element of $\mathcal{E}(\mathcal{H})$ contains another. A vertex of $\mathcal{H}$ is called \textbf{isolated} if it is not contained in any edge of $\mathcal{H}$.
}
\end{definition}

\begin{definition}
{\rm Assume that $\mathcal{H}$ is a hypergraph. For any vertex $x\in V(\mathcal{H})$, $\mathcal{H}\setminus x$ is a hypergraph with vertex set $V(\mathcal{H})\setminus\{x\}$ and  edge set $\{E\in \mathcal{E}(\mathcal{H}):\ x\notin E\}$.
Moreover $\mathcal{H}/x$ is a hypergraph with vertex set  $V(\mathcal{H})\setminus\{x\}$ whose edges are the non-empty minimal elements (with respect to inclusion) of the set $\{E\setminus \{x\} \ : \ E\in \mathcal{E}(\mathcal{H})\}$. It is clear that $\mathcal{H}\setminus x$ and $\mathcal{H}/x$ are two simple hypergraphs.
They are called \textbf{deletion} and \textbf{contraction} of $\mathcal{H}$ by $x$, respectively.}
\end{definition}
Note that for a vertex $x\in V(\mathcal{H})$, $\del_{\Delta_\mathcal{H}}(x)=\Delta_{\mathcal{H}\setminus x}$ and $\lk_{\Delta_{\mathcal{H}}}(x)=\Delta_{\mathcal{H}/x}$.

\begin{definition}
{\rm Given a hypergraph $\mathcal{H}$, there are some notions of induced subgraph. Although we need just two of them, but for completeness of the context we bring all of them here. Given a subset $A$
of vertices, a \textbf{subhypergraph} on $A$ is the hypergraph $$\mathcal{H}_A = (A, \{E_i \cap A  : E_i \cap A\neq \emptyset \}).$$
Note that the new index set for edges is $\{i \in I | E_i \cap A \neq \emptyset\}$. A \textbf{vertex section
hypergraph} on $A$ is the hypergraph $$\mathcal{H}\times A = (A, \{E_i : E_i \subseteq A\}).$$

Given a subset $J\subseteq I$, let $\mathcal{E}_J=\{E_j : j\in J\}$, we let $\mathcal{H}_J=(V, \mathcal{E}_J)$ denote the \textbf{partial hypergraph} and the \textbf{edge section hypergraph} $\mathcal{H}\times J$ is a hypergraph which has the edge set $\mathcal{E}_J$ and the vertex set $\bigcup_{j\in J}E_j$.
}
\end{definition}

\begin{example}
{\rm Let $V = \{1, 2, 3, 4\}$, and consider
a hypergraph $\mathcal{H}$ with edges $E_1 = \{1, 2, 3\}$ and $E_2 = \{2, 4\}$. Then the subhypergraph of $\mathcal{H}$
induced by the vertex set $A = \{2, 3, 4\}$ has edges $E_1 \cap A = \{2, 3\}$ and $E_2 = \{2, 4\}$, while
the vertex section hypergraph $\mathcal{H} \times A$ only has the edge $E_2 = \{2, 4\}$.

Consider the hypergraph $\mathcal{H}'$ on $V$ with edges $E_1'=\{1, 2\}$, $E'_2 = \{2, 3\}$ and $E_3'=\{3,4\}$. Then the partial hypergraph of $\mathcal{H}'$ induced by $E_1', E_2'$ has the vertex set $V$, while the edge section hypergraph of $\mathcal{H}'$ induced by $E_1', E_2'$ has the vertex set $\{1,2,3\}$.}
\end{example}

\begin{definition}
{\rm A \textbf{chain} in $\mathcal{H}$ is a sequence $v_0, E_1, v_1, \dots , E_k, v_k$, where $v_i \in E_i$ for $1\leq i \leq k, v_i \in E_{i+1}$ for
$0 \leq i \leq k-1$, and $E_1, \dots , E_k$ are edges $\mathcal{H}$. For our convenience, we denote this chain by $E_1, \dots, E_k$, if there is no ambiguity. If the edges are all distinct, we obtain
a \textbf{path} of length $k$. If $k > 2$ and $v_0 = v_k$, we call the path a \textbf{cycle} of length $k$ or a \textbf{$k$-cycle} and we denote it by $C_k$. We say that $\mathcal{H}$ is \textbf{$C_k$-free} if it doesn't contain any cycle $C_k$ as an edge section hypergraph.}
\end{definition}

\begin{definition}
{\rm A $d$-uniform hypergraph $\mathcal{H}$ is called \textbf{strongly connected} if for each distinct edges $E$ and $E'$, there is a chain $E=E_0, E_1, \dots, E_{k-1}, E_k=E'$ of edges of $\mathcal{H}$ such that for each $i:=0, 1, \dots, k-1$, $|E_i\cap E_{i+1}|=d-1$.}
\end{definition}

\section{Matching numbers of hypergraphs}

In this section, firstly, inspired by the definition of an induced matching in \cite{MVi}, we introduce the concepts of induced matching number and
 semi induced matching number of a hypergraph. Then we give some equalities and inequalities between these invariants.
\begin{definition}
{\rm A set $\{E_1, \dots, E_k\}$ of edges of a hypergraph $\mathcal{H}$ is called a \textbf{semi induced matching} if the only edges contained in $\bigcup_{\ell=1}^kE_\ell$ are $E_1, \dots, E_k$. A semi induced matching which all of its elements are mutually disjoint is called an \textbf{induced matching}. Also, we set
$$c_{\mathcal{H}}:=\max\{|\bigcup_{\ell=1}^kE_\ell|-k \ : \ \{E_1, \dots, E_k\} \ \T{is \ an \ induced \ matching \ in \ }\mathcal{H}\},$$
$$c'_{\mathcal{H}}:=\max\{|\bigcup_{\ell=1}^kE_\ell|-k \ : \ \{E_1, \dots, E_k\} \ \T{is \ a \ semi \ induced \ matching \ in \ }\mathcal{H}\},$$
and we call them \textbf{induced matching number} and \textbf{semi induced matching number} of $\mathcal{H}$, respectively.}
\end{definition}

The following theorem compares the invariants $c_{\mathcal{H}}$, $c'_{\mathcal{H}}$ and
$\T{dim}(\Delta_{\mathcal{H}})$ for an arbitrary hypergraph $\mathcal{H}$.

\begin{theorem}\label{dim}
For any hypergraph $\mathcal{H}$, we have the following inequalities.
$$c_{\mathcal{H}}\leq c'_{\mathcal{H}}\leq \T{dim}(\Delta_{\mathcal{H}})+1$$
\end{theorem}
\begin{proof}
It is clear that every induced matching of $\mathcal{H}$ is a semi induced matching. So, we have $c_{\mathcal{H}}\leq c'_{\mathcal{H}}$. To prove the last inequality, suppose that $\{E_1, \dots, E_k\}$ is a semi induced matching in $\mathcal{H}$ such that $c'_{\mathcal{H}}=|\bigcup_{\ell=1}^kE_\ell|-k$. Set $S_0=\emptyset$ and for each $1\leq i\leq k$, if $E_i\cap S_{i-1}\neq \emptyset$, then set $S_i=S_{i-1}$; else, choose a vertex $x_i\in E_i$ and set $S_i=S_{i-1}\cup \{x_i\}$. Now, consider the set $G=(\bigcup_{\ell=1}^kE_\ell)\setminus S_k$. We claim that $G$ is an independent set of vertices in $\mathcal{H}$.
By contrary, assume that $E\subseteq G$ for some $E\in \mathcal{E}(\mathcal{H})$. Then $E\cap S_k=\emptyset$ and $E\subseteq \bigcup_{\ell=1}^kE_\ell$. So $E=E_i$ for some $1\leq i\leq k$, since $\{E_1, \dots, E_k\}$ is a semi induced matching in $\mathcal{H}$. From the choice of $x_i$s, it is clear that $x_j\in E_i\cap S_k$ for some $1\leq j\leq i$, which is a contradiction.
Therefore, $G$ is contained in a facet $F$ of $\Delta_{\mathcal{H}}$. Since $|S_k|\leq k$, we have $c'_{\mathcal{H}}\leq |G|\leq |F|\leq \T{dim}(\Delta_{\mathcal{H}})+1$, which completes the proof.
\end{proof}

The following example illustrates that the  inequalities in Theorem \ref{dim} can be strict.
\begin{example}\label{ex}
{\rm Let $\mathcal{H}$ be a hypergraph with vertex set $V=\{x_1, \dots, x_6\}$ and edges $E_1=\{x_1, x_2, x_3\}, E_2=\{x_2, x_3, x_4\}$ and $E_3=\{x_4, x_5, x_6\}$. Then one can see that $c_{\mathcal{H}}=2$ and $c'_{\mathcal{H}}=3$. So $c_{\mathcal{H}}<c'_{\mathcal{H}}$.

Assume that $G$ is a star graph with vertex set $V=\{x_1, \dots, x_4\}$ and edges $\{x_1, x_2\}, \{x_1, x_3\}, \{x_1, x_4\}$. Then one can easily see that $c'_G=1$, but $\T{dim}(\Delta_G)=2$. So, even when $\mathcal{H}$ is a graph, the second inequality in Theorem \ref{dim} can be strict.}
\end{example}

\begin{remark}\label{2.4}
{\rm It is easily seen that when $\mathcal{H}$ is a graph, $c_{\mathcal{H}}$ is the well-known
induced matching number of $\mathcal{H}$; i.e. the maximum number of $3$-disjoint edges in $\mathcal{H}$. Also, it can be easily seen that for a hypergraph $\mathcal{H}$, a subset $\{E_1, \dots, E_k\}$ of edges of $\mathcal{H}$ is a semi induced matching if the edges of the vertex section hypergraph $\mathcal{H} \times \bigcup_{i=1}^k E_i$ are exactly $E_1, \dots, E_k$. So $c'_\mathcal{H}$ can be defined as
$$\max \{|V(\mathcal{K})|-|\mathcal{E}(\mathcal{K})| : \mathcal{K} \T{ \ is \ a \ vertex \ section \ hypergraph \ of \ } \mathcal{H} \T{ \ with \ no \ isolated \ vertex}\}.$$}
\end{remark}

In the following proposition, we provide conditions under which $c_{\mathcal{H}}= c'_{\mathcal{H}}$.
\begin{proposition}\label{MH}
Assume that $\mathcal{H}$ is a $d$-uniform hypergraph such that for each distinct edges $E$ and $E'$, $E\cap E'\neq \emptyset$ implies that $|E\cap E'|=d-1$. Then $c_{\mathcal{H}}= c'_{\mathcal{H}}$.
\end{proposition}
\begin{proof}
In view of Theorem \ref{dim}, it is enough to show that $c'_{\mathcal{H}}\leq c_{\mathcal{H}}$. To this end, assume that $\{E_1, \dots, E_k\}$ is a semi induced matching in $\mathcal{H}$ such that $ |\bigcup_{\ell=1}^kE_\ell|-k=c'_\mathcal{H}$. It is sufficient to show that there is a subset $S$ of $\{1, \dots, k\}$ such that $\{E_\ell \ : \ \ell\in S\}$ is an induced matching in $\mathcal{H}$ and
$$ |\bigcup_{\ell=1}^kE_\ell|-k \leq|\bigcup_{\ell\in S}E_\ell|-|S|.$$
We use induction on $k$. The result is clear when $k=1$. So assume inductively that $k>1$ and the result is true for smaller values of $k$. We may consider the following cases.

\textbf{Case I.} Suppose that there is an integer $1\leq i\leq k$ such that $E_i\cap (\bigcup_{\ell =1,\ell\neq i}^k E_\ell)=\emptyset$. Then by inductive hypothesis, there is a subset $S$ of $\{1, \dots, i-1, i+1, \dots , k\}$ such that $\{E_\ell \ : \ \ell\in S\}$ is an induced matching in $\mathcal{H}$ and we have
$$ |\bigcup_{\ell=1, \ell\neq i}^kE_\ell|-(k-1) \leq |\bigcup_{\ell\in S}E_\ell|-|S|.$$
Now, set $S'=S\cup \{i\}$. It is obvious that $\{E_\ell \ : \ \ell\in S'\}$ is an induced matching in $\mathcal{H}$ and we have
\begin{align*}
|\bigcup_{\ell=1}^kE_\ell|-k&=|\bigcup_{\ell=1,\ell\neq i}^k E_\ell|-(k-1)+|E_i|-1\\
&\leq|\bigcup_{\ell\in S}E_\ell|-|S|+|E_i|-1\\
&=|\bigcup_{\ell\in S'}E_\ell|-|S'|
\end{align*}
as desired.

\textbf{Case II.} Suppose that there is an integer $1\leq i\leq k$ such that $0<|E_i\cap (\bigcup_{\ell=1, \ell\neq i}^kE_\ell)|< |E_i|.$ Then inductive hypothesis implies that there is a subset $S$ of $\{1, \dots, i-1, i+1, \dots , k\}$ such that $\{E_\ell \ : \ \ell\in S\}$ is an induced matching in $\mathcal{H}$ and
$$ |\bigcup_{\ell=1,\ell\neq i}^kE_\ell|-(k-1) \leq|\bigcup_{\ell\in S}E_\ell|-|S|.$$
On the other hand,  by our assumption on $\mathcal{H}$, we should have $|E_i\cap (\bigcup_{\ell=1,\ell\neq i}^kE_\ell)|=d-1$. Now, we have
\begin{align*}
|\bigcup_{\ell=1}^kE_\ell|-k&=|\bigcup_{\ell=1,\ell\neq i}^kE_\ell|-(k-1)+|E_i|-|E_i\cap (\bigcup_{\ell=1, \ell\neq i}^k E_\ell)|-1\\
&\leq|\bigcup_{\ell\in S}E_\ell|-|S|+d-(d-1)-1\\
&=|\bigcup_{\ell\in S}E_\ell|-|S|
\end{align*}
as desired.

\textbf{Case III.} Suppose that for each $1\leq i\leq k$, $E_i\subseteq \bigcup_{\ell=1, \ell\neq i}^kE_\ell$. Then by inductive hypothesis, there is a subset $S$ of $\{1, \dots, k-1\}$ such that $\{E_\ell \ : \ \ell\in S\}$ is an induced matching in $\mathcal{H}$ and $$ |\bigcup_{\ell=1}^{k-1}E_\ell|-(k-1) \leq|\bigcup_{\ell\in S}E_\ell|-|S|.$$
So, we have
\begin{align*}
|\bigcup_{\ell=1}^kE_\ell|-k&=|\bigcup_{\ell=1}^{k-1}E_\ell|-(k-1)-1\\
&\leq|\bigcup_{\ell\in S}E_\ell|-|S|-1\\
&\leq|\bigcup_{\ell\in S}E_\ell|-|S|
\end{align*}
as desired.
\end{proof}

Although by Remark \ref{2.4}, one may find out that for a graph $G$, $c_G=c'_G$, but this is an immediate consequence of Proposition \ref{MH} as follows.
\begin{corollary}
For a simple graph $G$, we have $c_G=c'_G$.
\end{corollary}

The concept of matching number of a hypergraph is known as a generalization of one in graph theory (see \cite{Berge}). In fact, the maximum number of mutually disjoint edges of a hypergraph is called the matching number. H$\grave{a}$ and Van Tuyl in \cite{HT1} showed that when $\mathcal{H}$ is a graph, its matching number is an upper bound for $\T{reg}(R/I_{\Delta_{\mathcal{H}}})$. Here, by benefitting their work, we improve the definition of matching number of a hypergraph so that we can generalize this result to special class of hypergraphs (see Remark \ref{mH}). So, we present a new definition for matching number of a hypergraph as follows.

\begin{definition}(Compare \cite[Definition 6.6]{HT1}.)
{\rm  A set of edges of a hypergraph $\mathcal{H}$ is called a \textbf{matching} if they are pairwise disjoint. Also, we set
 $$m_{\mathcal{H}}:=\max\{|\bigcup_{\ell=1}^kE_\ell|-k \ : \ \{E_1, \dots, E_k\} \ \T{is \ a \ matching \ in \ }\mathcal{H}\},$$
and we call it the \textbf{matching number} of $\mathcal{H}$.}
\end{definition}
One can see that this definition is a natural generalization of one in graph theory, i.e. when $\mathcal{H}$ is a graph, $m_{\mathcal{H}}$ is the largest size of a maximal matching in $\mathcal{H}$. Furthermore, it is obvious that  $c_{\mathcal{H}}\leq m_{\mathcal{H}}$ for any hypergraph $\mathcal{H}$. Although, at one look, no relation can be seen between $c'_{\mathcal{H}}$ and $m_{\mathcal{H}}$, but Proposition \ref{MH} shows that $c'_G\leq m_G$, for special class of hypergraphs consisting simple graphs.  Note that the mentioned condition in Proposition \ref{MH} is different from the property of \emph{strongly connected} for hypergraphs.

\section{Regularity of edge ideal of certain hypergraphs}
In this section, we show that for a hypergraph $\mathcal{H}$, the introduced invariants in Section 2 give bounds for $\T{reg}(R/I_{\Delta_{\mathcal{H}}})$ and for some families of hypergraphs we give the precise amount of $\T{reg}(R/I_{\Delta_{\mathcal{H}}})$ in terms of these numbers. We begin by the following remark.

\begin{remark}\label{mH}
{\rm Morey and Villarreal in \cite{MVi} showed that $c_{\mathcal{H}}$ is a lower bound for $\T{reg}(R/I_{\Delta_{\mathcal{H}}})$ for a simple hypergraph $\mathcal{H}$. Hereafter, we are trying to find circumstances under which $c'_{\mathcal{H}}$ or $m_{\mathcal{H}}$ is an upper bound for $\T{reg}(R/I_{\Delta_{\mathcal{H}}})$. Note that in the light of \cite[Theorem 6.7]{HT1}, $m_{\mathcal{H}}$ is an upper bound for $\T{reg}(R/I_{\Delta_{\mathcal{H}}})$, where $\mathcal{H}$ is a simple graph. But we may have this result for more hypergraphs. In this regard, recall that a subset $C$ of the edges of a hypergraph $\mathcal{H}$ is called a \textbf{2-collage} for $\mathcal{H}$ if for each edge $E$ of $\mathcal{H}$ we can delete a vertex $v$ so that $E\setminus \{v\}$ is contained in some edge of $C$. Hence if $\mathcal{H}$ is a $d$-uniform hypergraph such that for each distinct edges $E$ and $E'$, $E\cap E'\neq \emptyset$ implies that $|E\cap E'|=d-1$, one can easily see that any maximal matching in $\mathcal{H}$ is a 2-collage. So, in view of \cite[Corollary 3.9(a)]{MVi} and \cite[Theorem 1.2]{HW}, one can have
$$c_{\mathcal{H}}\leq \T{reg}(R/I_{\Delta_{\mathcal{H}}})\leq m_{\mathcal{H}}.$$}
\end{remark}

As a main result of this paper, we are going to show that $c'_{\mathcal{H}}$ is an upper bound for $\T{reg}(R/I_{\Delta_{\mathcal{H}}})$ for a certain class of hypergraphs. To this end, we need to recall the following definition.

\begin{definition}
{\rm Let $\Delta$ be a simplicial complex on the vertex set $V =
\{x_1,\ldots, x_n\}$. Then $\Delta$ is \textbf{vertex decomposable}
if either:

1) The only facet of $\Delta$ is $\{x_1,\ldots, x_n\}$, or
$\Delta=\emptyset$.

2) There exists a vertex $x\in V$ such that $\del_{\Delta}(x)$ and
$\lk_{\Delta}(x)$ are vertex decomposable, and such that every facet
of $\del_{\Delta}(x)$ is a facet of $\Delta$.}
\end{definition}

A vertex $x\in V$ for which every facet
of $\del_{\Delta}(x)$ is a facet of $\Delta$ is called a
\textbf{shedding vertex} of $\Delta$. Note that this is equivalent to say that no facet of $\lk_{\Delta}(x)$ is a facet of
$\del_{\Delta}(x)$.

A hypergraph $\mathcal{H}$ is called vertex decomposable, if the independence
complex $\Delta_{\mathcal{H}}$ is vertex decomposable and a vertex of $\mathcal{H}$ is called a shedding vertex if it is a shedding vertex of $\Delta_{\mathcal{H}}$. It is easily seen that if $x$ is a shedding vertex of $\mathcal{H}$ and $\{E_1, \dots, E_k\}$ is the set of all edges of $\mathcal{H}$ containing $x$, then every facet of $\mathcal{H}\setminus x$ contains $E_i\setminus \{x\}$ for some $1\leq i\leq k$.

For our main result we also need to illustrate the relations between $c'_{\mathcal{H}}$, $c'_{\mathcal{H}\setminus x}$ and $c'_{\mathcal{H}/x}$ for a vertex $x$ of $\mathcal{H}$. Note that it is obvious that $c_{\mathcal{H}\setminus x}\leq c_{\mathcal{H}}$ and $c'_{\mathcal{H}\setminus x}\leq c'_{\mathcal{H}}$. Now, suppose that $\{E_1\setminus\{x\}, \dots, E_k\setminus\{x\}\}$ is a semi induced matching in $\mathcal{H}/x$ such that  $c'_{\mathcal{H}/x}=|\bigcup_{\ell=1}^k (E_\ell\setminus\{x\})|-k$. The following example shows that it is not necessarily true that $\{E_1, \dots, E_k\}$ is a semi induced matching in $\mathcal{H}$.

\begin{example}
{\rm Let $\mathcal{H}$ be a hypergraph with $V(\mathcal{H})=\{x_1, \dots, x_5\}$ and $\mathcal{E}(\mathcal{H})=\{E_1=\{x_1, x_2, x_3\}, E_2=\{x_2, x_3, x_4\}, E_3= \{x_4, x_5\}\}$. Then $\mathcal{E}(\mathcal{H}/x_1)=\{E_1\setminus\{x_1\}, E_3\setminus \{x_1\}\}$. It is clear that $\{E_1\setminus\{x_1\}, E_3\setminus \{x_1\}\}$ is a semi induced matching in $\mathcal{H}/x_1$ but $\{E_1, E_3\}$ is not a semi induced matching in $\mathcal{H}$.}
\end{example}

Now, the following two lemmas provide conditions under which we can get to a semi induced matching in $\mathcal{H}$ from one in $\mathcal{H}/x$, for a vertex $x$ of $\mathcal{H}$.

\begin{lemma}\label{1}
Assume that $\mathcal{H}$ is a $C_2$-free hypergraph, $x$ is a vertex of $\mathcal{H}$
 and $k$ is the smallest integer such that there exists a semi induced matching $\{E_1\setminus\{x\}, \dots, E_k\setminus\{x\}\}$ in $\mathcal{H}/x$ so that
 $c'_{\mathcal{H}/x}=|\bigcup_{\ell=1}^k (E_\ell\setminus\{x\})|-k$. Then $\{E_1, \dots, E_k\}$ is a semi induced matching in $\mathcal{H}$ and so if $x\in E_i$ for some $1\leq i\leq k$, we have $c'_{\mathcal{H}/x}+1\leq c'_{\mathcal{H}}$.
\end{lemma}
\begin{proof}
Suppose that there is an edge $E$ of $\mathcal{H}$ such that $E\subseteq\bigcup_{\ell=1}^k E_\ell$. Then $E\setminus \{x\}\subseteq\bigcup_{\ell=1}^k (E_\ell\setminus \{x\})$. Now, we have three cases:

\textbf{Case I.} If $x\in E$, then $E\setminus \{x\}=E_i\setminus \{x\}$, for some $1\leq i\leq k$. If $x\not\in E_i$, then $E$ strictly contains $E_i$ which is a contradiction. So, $x\in E_i$ and hence $E=E_i$ as desired.

\textbf{Case II.} If $x\not\in E$ and $E$ is an edge of $\mathcal{H}/x$, then $E=E_i\setminus \{x\}$, for some $1\leq i\leq k$. If $x\in E_i$, then $E_i$ strictly contains $E$ which is a contradiction. So, $x\not\in E_i$ which implies that $E=E_i$ as desired.

\textbf{Case III.} If $x\not\in E$ and $E$ is not an edge of $\mathcal{H}/x$, then there is an edge $E'$ of $\mathcal{H}$ containing $x$ such that $E'\setminus \{x\}\subset E$ and $E'\setminus \{x\}$ is an edge of $\mathcal{H}/x$. So, $E\cap E'=E'\setminus\{x\}$. Since $\mathcal{H}$ is $C_2$-free, $|E'\setminus\{x\}|=1$. Since $E'\setminus \{x\}\subseteq \bigcup_{\ell=1}^k (E_\ell\setminus \{x\})$, then $E'\setminus \{x\}=E_i\setminus \{x\}$ for some $1\leq i\leq k$.
Thus $|E_i\setminus\{x\}|=1$. Moreover, $E_i\setminus\{x\}\nsubseteq \bigcup_{\ell=1,\ell\neq i}^k (E_\ell\setminus \{x\})$, since otherwise
$E_i\setminus\{x\}\subseteq E_j\setminus\{x\}$ for some $j\neq i$, which is impossible. Therefore, $\{E_\ell\setminus \{x\}, 1\leq \ell\leq k, \ell\neq i\}$
is a semi inducing matching in $\mathcal{H}/x$ and $|\bigcup_{\ell=1,\ell\neq i}^k (E_\ell\setminus\{x\})|-(k-1)=|\bigcup_{\ell=1}^k (E_\ell\setminus\{x\})|-1-(k-1)=c'_{\mathcal{H}/x}$, which contradicts to our assumption on $k$. So this case can't occur.

Hence, $\{E_1, \dots, E_k\}$ is a semi induced matching in $\mathcal{H}$. Now, if $x\in E_i$ for some $1\leq i\leq k$, we have $$c'_{\mathcal{H}/x}=|\bigcup_{\ell=1}^k (E_\ell\setminus\{x\})|-k=|\bigcup_{\ell=1}^k E_\ell|-k-1\leq c'_{\mathcal{H}}-1,$$
which completes the proof.
\end{proof}

\begin{lemma}\label{2}
Assume that $\mathcal{H}$ is a $(C_2, C_5)$-free hypergraph, $x$ is a shedding vertex of $\mathcal{H}$
 and $\{E_1\setminus\{x\}, \dots, E_k\setminus\{x\}\}$ is a semi induced matching in $\mathcal{H}/x$ such that $x\not\in E_\ell$ for all $1\leq \ell \leq k$. Then there is an edge $F$ of $\mathcal{H}$ containing $x$ such that $\{E_1, \dots, E_k, F\}$ is a semi induced matching in $\mathcal{H}$. Moreover,
 $c'_{\mathcal{H}/x}+1\leq c'_{\mathcal{H}}$.
\end{lemma}
\begin{proof}
Let $\{F_1, \dots, F_s\}$ be the set of all edges containing $x$ and suppose, in contrary, that for each $F_i$, there is an edge $F_i'$ of $\mathcal{H}$ such that $F_i'\not\in \{E_1, \dots, E_k, F_i\}$, $F_i'\cap F_i\neq\emptyset$ and $F_i'\setminus F_i\subseteq \bigcup_{\ell=1}^kE_\ell$. Note that if $F_i'=F_j$ for some $1\leq j\leq s$, then $F_j\setminus F_i\subseteq\bigcup_{\ell=1}^kE_\ell$. Since $\mathcal{H}$ is $C_2$-free, then $F_i\cap F_j=\{x\}$ and $F_j\setminus F_i=F_j\setminus \{x\}$. So $F_j\setminus \{x\}\subseteq\bigcup_{\ell=1}^kE_\ell$. This is a contradiction, since $F_j\setminus \{x\}\in \mathcal{E}(\mathcal{H}/x)$.

Also, note that for each distinct integers $1\leq i,j\leq s$, $F_i'\neq F_j'$. Because, if we have $F_i'=F_j'$ for some distinct integers $1\leq i,j\leq s$, then we should have $F_i'\setminus (F_i\cap F_j)\subseteq \bigcup_{\ell=1}^kE_\ell$. On the other hand, we know that $x\not\in F_i'$ and since $\mathcal{H}$ is $C_2$-free, $F_i\cap F_j=\{x\}$. Hence, $F_i'\subseteq \bigcup_{\ell=1}^kE_\ell$. So there exists an edge $E\setminus\{x\}\in \mathcal{E}(\mathcal{H}/x)$ such that $E\setminus\{x\}\subseteq F'_i\subseteq \bigcup_{\ell=1}^kE_\ell$, which is a contradiction.

Moreover, note that for each distinct integers $1\leq i,j\leq s$, $F_j\cap F'_j\not\subseteq F'_i\setminus F_i$, because otherwise since $F'_i\setminus F_i$ and $F'_j\setminus F_j$ are contained in $\bigcup_{\ell=1}^k E_\ell$, we should have $F'_j\subseteq \bigcup_{\ell=1}^k E_\ell$, which is a contradiction. Hence, $E_\ell\cap F'_i\neq F_j\cap F'_j$ for all $1\leq \ell\leq k$.

Now, set $S=\bigcup_{i=1}^s(F_i'\setminus F_i)$. At first, we are going to show that $S$ is an independent set of vertices in $\mathcal{H}/x$. Suppose, in contrary, that $S$ is not independent. Then, since $S\subseteq \bigcup_{\ell=1}^kE_\ell$ and $\{E_1, \dots, E_k\}$ is a semi induced matching in $\mathcal{H}/x$, there should exist an $E_\ell$ which intersects with two distinct edges $F_i'$ and $F_j'$. So, since $\mathcal{H}$ is $C_2$-free, $E_\ell - F_i' - F_i - F_j - F_j' - E_\ell$ forms a subhypergraph $C_5$ in $\mathcal{H}$ which is a contradiction. Thus, $S$ is an independent set of vertices in $\mathcal{H}/x$.
We extend $S$ to a facet $G$ of $\Delta_{\mathcal{H}/x}$. $G$ is also a facet of $\Delta_{\mathcal{H}\setminus x}$; because otherwise $G$ is contained in a facet $K$ of $\Delta_{\mathcal{H}\setminus x}$. Now, since $x$ is a shedding vertex, $K$ contains $F_i\setminus \{x\}$ for some $1\leq i\leq s$. Hence, $F'_i\subseteq K$, because $F'_i\setminus F_i\subseteq S\subseteq G\subseteq K$ and $x\not\in F'_i$. This is a contradiction, since $F'_i\in \mathcal{E}(\mathcal{H} \setminus x)$.
So we found a facet of $\Delta_{\mathcal{H}/x}$, which is a facet of $\Delta_{\mathcal{H}\setminus x}$. But this contradicts to the fact that $x$ is a shedding vertex.
 So, we proved that $\{E_1, \dots, E_k, F_i\}$ is a semi induced matching in $\mathcal{H}$, for some edge $F_i$ containing $x$. Now, let
  $c'_{\mathcal{H}/x}=|\bigcup_{\ell=1}^k(E_\ell\setminus\{x\})|-k$. Since $F_i\setminus \{x\}\nsubseteq \bigcup_{\ell=1}^kE_\ell$, $c'_{\mathcal{H}/x}+1\leq |(\bigcup_{\ell=1}^kE_\ell)\cup F_i|-(k+1)\leq c'_{\mathcal{H}}$ as required.
\end{proof}

Now, we are ready to state our main result of this section.
\begin{theorem}\label{reg}
Let $\mathcal{H}$ be a $(C_2,C_5)$-free vertex decomposable hypergraph.Then
$$\T{reg}(R/I_{\Delta_{\mathcal{H}}})\leq c'_{\mathcal{H}}\leq \T{dim}(\Delta_{\mathcal{H}})+1.$$
\end{theorem}
\begin{proof}
In the light of Theorem \ref{dim} it is enough to prove $\T{reg}(R/I_{\Delta_{\mathcal{H}}})\leq c'_{\mathcal{H}}$. In this regard, we use induction on $|V(\mathcal{H})|$. If $|V(\mathcal{H})|=2$, the result is clear. Suppose, inductively, that the result has been proved for smaller values of $|V(\mathcal{H})|$. Assume that $x$ is a shedding vertex of $\mathcal{H}$. Let $\Delta=\Delta_{\mathcal{H}}$, $\Delta_1=\Delta_{\mathcal{H}\setminus x}$ and $\Delta_2=\Delta_{\mathcal{H}/x}$. Then $\mathcal{H}\setminus x$ and $\mathcal{H}/x$ are $(C_2,C_5)$-free vertex decomposable hypergraphs and no facet of $\Delta_2$ is a facet of $\Delta_1$.
 By inductive hypothesis we have
$$\T{reg}(R/I_{\Delta_1})\leq c'_{\mathcal{H}\setminus x} \ \T{and} \ \T{reg}(R/I_{\Delta_2})\leq c'_{\mathcal{H}/x}.$$
On the other hand, we have the inequality
$$\T{reg}(R/I_{\Delta})\leq\max\{\T{reg}(R/I_{\Delta_1}),\T{reg}(R/I_{\Delta_2})+1\}.$$
Hence
$$\T{reg}(R/I_{\Delta})\leq \max\{c'_{\mathcal{H}\setminus x},c'_{\mathcal{H}/x}+1\}.$$
Now, the result immediately follows from Lemmas \ref{1} and \ref{2}.
\end{proof}

\begin{example}
{\rm Assume that $d\geq 3$ and $\mathcal{H}$ is a $d$-uniform simple hypergraph with vertex set $$V(\mathcal{H})=\bigcup_{i=1}^k \bigcup_{j=1}^{d-1}\{x_{i,j}\}\cup \{x\},$$
and edge set
$$\mathcal{E}(\mathcal{H})=\{E_i=\{x_{i,1}, \dots, x_{i,d-1}, x\} : 1\leq i\leq k\}.$$
one may find out that the only 2-collage of $\mathcal{H}$ is $\{E_1, \dots, E_k\}$. So, in the light of \cite[Theorem 1.2]{HW}, we have
$$\T{reg}(R/I_{\Delta_{\mathcal{H}}})\leq k(d-1).$$
On the other hand, One can easily check that $\mathcal{H}$ is a $(C_2,C_5)$-free vertex decomposable hypergraph and $\{E_1, \dots, E_k\}$ is the semi induced matching of $\mathcal{H}$ such that $c'_\mathcal{H}=|\bigcup_{i=1}^k E_i|-k$. Hence, by Theorem \ref{reg}, we have
$$\T{reg}(R/I_{\Delta_{\mathcal{H}}})\leq k(d-1)+1-k.$$}
\end{example}

The above example illustrates that for $d\geq 3$ and large enough values of $k$, the upper bound on $\T{reg}(R/I_{\Delta_{\mathcal{H}}})$ presented in \cite[Theorem 1.2]{HW} in this special case is much larger than the actual value of $\T{reg}(R/I_{\Delta_{\mathcal{H}}})$, and our upper bound in Theorem \ref{reg} is better than one given in \cite[Theorem 1.2]{HW}. Of course note that the proof of Lemma 3.4 in \cite{HW} has some flaws, and so the proof of Theorem 1.2 in \cite{HW} will be uncertain. The following example shows this defect.

\begin{example}
{\rm Assume that $\mathcal{H}$ is a hypergraph with vertex set $V=\{a,b,c,d,e,f\}$ and edge set $\{\{a,b\}, \{a,c\}, \{d,f\}, \{e,f\}, \{b,c,d,e\}\}$. With the notations in \cite{HW}, for each edge $E$ of $\mathcal{H}$, let $\mathcal{H}_E$ be the hypergraph whose edge set consists of the minimal (under inclusions) members of $\{E\cup E' : E'\neq E \T{ \ is \ an \ edge \ of \ } \mathcal{H}\}$. So, by considering $E=\{a,b\}$, the edge set of $\mathcal{H}_E$ is $\{\{a,b,c\},\{a,b,d,f\}, \{a,b,e,f\} \}$. Now, one can easily see that $\{E_0=\{b,c,d,e\}\}$ is a 2-collage for $\mathcal{H}$ but $\{E\cup E_0\}$ is not even an edge of $\mathcal{H}_E$ and it doesn't contain any 2-collage of $\mathcal{H}_E$. Also, for each choice of $E\neq E_0$, the above assertion holds. This shows the mentioned defect of the proof of \cite[Lemma 3.4]{HW}.}
\end{example}

Since $c_{\mathcal{H}}=c'_{\mathcal{H}}$, in conjunction with Theorem \ref{reg} and \cite[Corollary 3.9(a)]{MVi}, can characterize the regularity of the Stanley-Reisner ring $R/I_{\Delta_{\mathcal{H}}}$ precisely, this question arises that when the equality $c_{\mathcal{H}}=c'_{\mathcal{H}}$ holds? In the light of Remark \ref{mH}, the similar question can be asked about the equality $c_{\mathcal{H}}=m_{\mathcal{H}}$. With this point of view, \cite[Corollary 3.9(a)]{MVi},
Proposition \ref{MH}, Remark \ref{mH} and Theorems \ref{dim} and \ref{reg} imply the next corollary.
\begin{corollary}(Compare \cite[Theorem 1.9]{KM}.)
\begin{itemize}
\item[(i)] Assume that $\mathcal{H}$ is a $(C_2,C_5)$-free vertex decomposable hypergraph such that $c_{\mathcal{H}}=\T{dim}(\Delta_{\mathcal{H}})+1$. Then
$$\T{reg}(R/I_{\Delta_{\mathcal{H}}})= c_{\mathcal{H}}=c'_{\mathcal{H}}=\T{dim}(\Delta_{\mathcal{H}})+1.$$
\item[(ii)] If $\mathcal{H}$ is a $(C_2,C_5)$-free vertex decomposable hypergraph such that $c_{\mathcal{H}}=c'_{\mathcal{H}}$, then
$$\T{reg}(R/I_{\Delta_{\mathcal{H}}})= c_{\mathcal{H}}.$$
\item[(iii)] In particular, if $G$ is a simple $C_5$-free vertex decomposable graph, then
$$\T{reg}(R/I_{\Delta_G})= c_G.$$
\item[(iv)] Assume that $\mathcal{H}$ is a $d$-uniform hypergraph such that $c_{\mathcal{H}}=m_{\mathcal{H}}$ and for each distinct edges $E$ and $E'$, $E\cap E'\neq \emptyset$ implies that $|E\cap E'|=d-1$. Then
$$\T{reg}(R/I_{\Delta_{\mathcal{H}}})= c_{\mathcal{H}}=c'_{\mathcal{H}}=m_{\mathcal{H}}.$$
\end{itemize}
\end{corollary}


\providecommand{\bysame}{\leavevmode\hbox
to3em{\hrulefill}\thinspace}

\end{document}